\documentclass[12pt]{amsart}

\usepackage{amsmath,amssymb,amsthm}
\usepackage{graphicx}
\usepackage{color}
\usepackage[margin=30mm]{geometry}

\usepackage{lineno} 
\usepackage{color}

\usepackage{enumerate}
\usepackage{url}
\usepackage{hyperref}
\usepackage{breakurl}
\definecolor{darkblue}{rgb}{0,0,.75}
\definecolor{darkgreen}{rgb}{0,0.5,0}
\hypersetup{colorlinks=true, breaklinks=true, linkcolor=blue, citecolor=darkblue, menucolor=darkblue, urlcolor=darkblue}

\usepackage[round]{natbib}
\usepackage{doi}

\newtheorem{theorem}{Theorem}[section]

\newtheorem{proposition}[theorem]{Proposition}

\newtheorem{corollary}[theorem]{Corollary}

\theoremstyle{definition}

\theoremstyle{remark}
\newtheorem{remark}[theorem]{Remark}

\numberwithin{equation}{section}

\newcommand{\ex}{\mathbf {E}}
\newcommand{\pr}{\mathbf {P}}
\newcommand{\R}{\mathbb R}
\newcommand{\Rd}{\mathbb R^d}

\newcommand{\spctim}{\mathbb R^{d+1}}
\newcommand{\del}{\partial }
\newcommand{\1}{\mathbf 1}

\author{Boris Baeumer \and Peter Straka}
\address{Boris Baeumer, Department of Mathematics \& Statistics, University of Otago, New Zealand}
\email{bbaeumer@maths.otago.ac.nz}
\address{Peter Straka, School of Mathematics and Statistics,
UNSW Australia,
Sydney, NSW 2052, Australia}
\email{p.straka@unsw.edu.au}

\begin{document}


\title{Fokker--Planck and Kolmogorov Backward Equations
for Continuous Time Random Walk Scaling Limits}




\begin{abstract}
It is proved that the distributions of scaling limits of Continuous Time Random Walks (CTRWs)
solve integro-differential equations akin to Fokker--Planck Equations for diffusion processes.
In contrast to previous such results, it is not assumed that the
underlying process has absolutely continuous laws.
Moreover, governing equations in the backward variables are
derived.
Three examples
of anomalous diffusion processes illustrate the theory. 
\end{abstract}
\maketitle


\section{Introduction}

Continuous time random walks (CTRWs) are random walks with random waiting times $W_k$ between jumps $J_k$.  They have been applied in physics to a variety of systems exhibiting ``anomalous diffusion,''
with heavy-tailed waiting times leading to subdiffusive processes
whose variance grows $\propto t^\beta$, $0 < \beta < 1$,
and with heavy-tailed jumps leading to superdiffusive processes
which exhibit a faster scaling than Brownian motion
\citep{Metzler2000}. For a variety of applications, see e.g.\
\cite{Berkowitz06,Henry2000,FedotovIomin,Raberto2002,SchumerMIM}.
Scaling limits of CTRWs are non-Markovian time-changes of $\R^d$-valued Markov processes \citep{limitCTRW,Kolokoltsov09,Kobayashi2010a}.

The main tool for the analysis and computation of the distribution of CTRW limits is the (fractional) Fokker--Planck equation (FPE; considered here as a synonym with Kolmogorov Forward Equation) \citep{BMK00,Langlands2005a};
for textbooks with a quick introduction to fractional derivatives
see e.g.\ \cite{MeerschaertSikorskii} or
\cite{kolokoltsov2011markov}.
Governing FPEs have been derived in the literature, whose
solutions can be roughly classified as follows:
\begin{itemize}
\item
Classical (strong) solutions
\citep{Kolokoltsov09,HKU10,kolokoltsov2011markov,Magdziarz2014,Nane2015},
where the derivation assumes that the underlying space-time Feller
process (see below) has (differentiable) probability densities;
\item
Solutions in Banach space settings, in the framework of fractional evolution equations
\citep{PrussBook,BajlekovaThesis,Baeumer2001,BMMST,umarov2015introduction}
where the derivation is based on the assumption that the coefficients do not vary in time;
\item
Mild solutions based on Fourier-Laplace transforms
\citep{coupleCTRW,triCTRW,Jurlewicz}.
These allow for a coupling between jumps and waiting times
but assume constant coefficients.
\end{itemize}
One aim of this paper is to unify the above results and to
derive governing FPEs without these restricting assumptions.

A further important analytical and computational tool for anomalous diffusion processes is the (fractional) Kolmogorov backward equation.
It may be used to calculate distributions of occupation times
and first passage times for anomalous diffusion processes \citep{Carmi2010}.
In groundwater hydrology, scaling limits of CTRWs model the spread of contaminants in an aquifer \citep{Berkowitz06,SchumerMIM},
and (non-fractional) Kolmogorov backward equations have already been used to model the distribution of pollutant sources and travel times \citep{Neupauer1999}.
A mathematical framework for CTRW scaling limits and fractional Kolmogorov backward equations would hence be applicable to problems in groundwater hydrology, but has yet to be established,
which is the second aim of this paper.

The following topics are not discussed in this article in order to maintain the focus
on governing equations, but they should be mentioned as they are closely related:
\begin{itemize}
\item
If each waiting time and jump pair $(W_k,J_k)$ is coupled, 
their order is important:
CTRWs assume that $W_k$ precedes $J_k$, whereas OCTRWs (overshooting CTRWs)
assume that $J_k$ precedes $W_k$. 
The scaling limits of these two processes
may be as different as having mutually disjoint supports for all $t > 0$ \citep{Jurlewicz,StrakaHenry}.
This paper focuses on limits of CTRWs.
\item
In our analysis, we do not assume that any stochastic process
admits a Lebesgue density, hence the FPE (Th~\ref{th:FFPE}) is
given on the Banach space of positive measures;
such a result is apparently new.
If the Feller process $(A_r,D_r)$ below admits a (suitably regular)
density, then the CTRW limit does so, too \citep{Magdziarz2014}.
In general, however, and in particular for the
three examples discussed in the last section, the existence of densities
is unconfirmed.
\item
Our analysis defines CTRW limits via a continuous mapping approach
(Theorem~\ref{th:cont-map}), and the underlying assumption is the convergence
of Feller jump processes to a Feller diffusion process with jumps. If the
sequence of Feller jump processes is specified, then the sequence of CTRWs is
also specified, which can be illuminating for applications and the simulation
of sample paths. We skip this content with a warning that 
convergence can be difficult to establish 
\citep{JacodShiryaev,kolokoltsov2011markov}.
\end{itemize}

\paragraph{Organization of this paper:}
In Section~\ref{sec:limits} below, CTRW scaling limits are introduced
in a very general setting.
Section~\ref{sec:banach} introduces the Banach space setting needed
for the derivation of Kolmogorov backward equation
(Section~\ref{sec:backward})
and the Fokker--Planck Equation (FPE, Section~\ref{sec:forward}).
Finally, Section~\ref{sec:examples} contains three examples from statistical
physics which illustrate the forward and backward governing
equations.

\section{Scaling limits of CTRWs} \label{sec:limits}

We introduce CTRW limit processes by closely following  \cite{SemiMarkovCTRW}:
Let $c > 0$ be a scaling parameter, and write $A^c(n)$ for the position after the $n$-th jump, and $D^c(n)$ for the time of the $n$-th jump.  We assume that after each jump, a CTRW is renewed.
More precisely, $(A^c(n+1),D^c(n+1))$ depends on the previous trajectory $(A^c(0),D^c(0)), \ldots, (A^c(n),D^c(n))$ only through the latest pair $(A^c(n),D^c(n))$; but this is equivalent to $\{(A^c(n),D^c(n))\}_{n \in \mathbb N_0}$ being a Markov chain with state space $\spctim$.
We assume that the sequence $D^c(n)$ is strictly increasing.

By setting
$\bar A^c (t) = A^c(\lfloor t\rfloor)$,
$\bar D^c(t) = D^c(\lfloor t\rfloor)$,
a Markov chain as above defines a trajectory
$[0,\infty) \ni t \mapsto (\bar A^c(t),\bar D^c(t)) \in \spctim$.
This trajectory can then be mapped to a CTRW trajectory
as follows:
Define the right-continuous inverse
$E^c(t) := \inf\{u: \bar D^c(u) > t\}$ of $\bar D^c$.
Write $\bar A^c_-$ for the left-continuous version of $\bar A^c$.
Then the CTRW trajectory is given by
$$X^c(t) = \bar A^c_- \circ \bar E^c_-(t+),$$
that is, by the right-continuous version of the composition of the two left-continuous processes
$\bar A^c_- \circ \bar E^c_-$ \citep[Lemma~3.5]{StrakaHenry}.
One may then exploit the Skorokhod continuity of this path mapping to obtain the CTRW scaling limit as $c \to \infty$:

\begin{theorem} \label{th:cont-map}
Suppose that as $c \to \infty$, we have the weak convergence
\begin{align} \label{eq:spctim-conv}
\left \lbrace \left (\bar A^c(\lfloor cr \rfloor),
\bar D^c(\lfloor cr \rfloor)\right )\right \rbrace_{r \ge 0}
\Rightarrow
\left \lbrace \left ( A_r, D_r\right )\right \rbrace_{r \ge 0}
\end{align}
in the $J_1$ topology on c\`adl\`ag paths in $\spctim$, where
$D_r$ is a.s.\ strictly increasing and unbounded.
Then we also have the weak convergence
$$\left \lbrace X^c(t)\right \rbrace_{t \in \R} \Rightarrow \left \lbrace X(t) \right \rbrace_{t \in \R}$$
in the $J_1$ topology on c\`adl\`ag paths in $\Rd$, where
\begin{align}\label{eq:CTRW-limit}
X(t) = A_- \circ E(t+),
\end{align}
$A_-$ denotes the left-continuous process $\{A(t-)\}_{t \ge 0}$
and $E(t) = \inf\{u: D_u > t\}$.
\end{theorem}

\begin{proof}
This theorem is a direct consequence of Proposition 2.3 in \cite{StrakaHenry}.
\end{proof}

We stress that $E(t)$ and $X(t)$ are in general not Markovian.

Due to the above theorem, the large class of possible CTRW limit
processes is hence essentially given by
\eqref{eq:CTRW-limit}
and an $\spctim$ valued process $(A_r,D_r)$ which is the weak
limit of a sequence of (continuous time) Markov chains, where $D_r$ is strictly increasing and unbounded.
Such processes contain the class of diffusion processes with jumps, in the sense of \cite{JacodShiryaev}.
Details on the convergence of Feller-jump processes to a Feller diffusion process with jumps as in \eqref{eq:spctim-conv} are e.g.\ in Theorem~IX.4.8 of the mentioned textbook, and
in \cite{kolokoltsov2011markov} with somewhat more specificity.

The idea that CTRWs are essentially random walks in space-time
was seemingly first introduced explicitly to CTRWs by
\cite{Weron2008}, and used in \cite{HLS10PRL} to derive a Fractional
Fokker-Planck Equation (FPE) with space- and time-dependent drift.
(For a more detailed derivation of the FPE, see \cite{Magdziarz2014}.)

The following scaling limits $(A_r,D_r)$ have been
considered in the literature:
Uncoupled and coupled stable limits \citep{limitCTRW,coupleCTRW},
triangular array limits \citep{triCTRW,Jurlewicz},
position-dependent, stable-like limits \citep{Kolokoltsov09}
and stochastic differential equations with diffusion component $A_r$ and
subordinator $D_r$ \citep{Weron2008,Magdziarz2014}.

To specify the class of space-time limit processes $(A_r,D_r)$, we first define the operator
$\mathcal A_0: C_0^2(\spctim) \to C_0(\spctim)$
(with Einstein notation)
by
\begin{multline}
\label{eq:DAgen}
\mathcal A_0 f(x,s)
=  b^i(x,s) \del_{x_i} f(x,s)
+ \gamma(x,s) \del_s f(x,s)
+\frac{1}{2}  a^{ij}(x,s) \del_{x_i} \del_{x_j} f(x,s)\\
+\int\limits_{z \in \Rd} \int\limits_{w \ge 0} \left[f(x+z,s+w)
-f(x,s)-  z^i \1(\|z\| < 1) \del_{x_i} f(x,t)\right]  K(x,s;dz,dw).
\end{multline}
We adopt the following basic conditions on the coefficients are:
for $i,j = 1, \ldots, d$ the mappings $(x,s) \mapsto b^i(x,s)$,
$(x,s) \mapsto a^{ij}(x,s)$,
$(x,s) \mapsto \gamma(x,s)$,
are in $C_b(\spctim)$,
the measures $K(x,s;\cdot, \cdot)$ are L\'evy measures for every $(x,s) \in \spctim$ and
$Kg(x,s) := \iint K(x,s;dz,dw)g(z,w)$ lies in $C_b(\spctim)$ for every
$g \in B_b(\spctim)$ (bounded measurable) which is $0$ in a neighbourhood of the origin \citep{JacodShiryaev}.
We note however that these conditions are not sufficient
for $\mathcal A_0$ to generate a Feller process;
for sufficient conditions, consult e.g.\ \citet[Ch~6]{Applebaum}.

We assume that $(A_r,D_r)$ is a Feller process with strongly continuous
semigroup $(T_r, r \ge 0)$ acting on $C_0(\spctim)$.
The infinitesimal generator $\mathcal A$  of $(T_r, r \ge 0)$ is
such that $C_0^2(\spctim) \subset \mathrm{Dom}(\mathcal A)$ and $\mathcal A f= \mathcal A_0f$ for all $f \in C_0^2(\spctim)$; for details, see e.g.\ Ch~6.7 in \cite{Applebaum}.
We write $\pr^{x,s}$ for the (canonical) probability measure induced by $(T_r, r \ge 0)$ and
$\pr^{x,s}(A_0 = x, D_0 = s) = 1$.
The requirement that $D_r$ be strictly increasing a.s.\ means that $\gamma(x,s) \ge 0$, that the diffusive component of $D_r$ vanishes, that
the measures $K(x,s;\cdot, \cdot)$ are supported on $\Rd \times [0,\infty)$, and that $\int_0^1wK(x,s,\Rd,dw)<\infty$.
Moreover, the truncation term in the integral does not apply to the $d+1$st coordinate.
For technical reasons, we require another, not very restrictive assumption:
\begin{description}
\item[Transience in the time-component]
If $f\in C_0(\spctim)$ has support $\mathrm{Supp}(f) \subset\Rd\times(-\infty,B]$ for some $B\in\R$, then the potential of $f$,
\begin{align} \label{ass:transience}
Uf(x,s)
= \int_0^\infty T_r f(x,s)\,dr
\end{align}
is a continuous function with $Uf(\cdot,s)\in C_0(\Rd)$ for all $s$; i.e., with a slight abuse of notation there exists a kernel $U$ such that
$$Uf(x,s)=\int U(x,s;dz,dw)f(z,w).$$
\end{description}
For example, if $D_r$ is a subordinator then this assumption is satisfied \citep{Bertoin04}. 
$U$ is commonly referred to as the potential kernel of the semigroup $(T_r, r \ge 0)$.


We can now give a result which characterises the distribution of $X_t$ for Lebesgue-almost every $t \in \R$:

\begin{theorem}\label{prop:law}
Let $H(x,s;v) := K(x,s;\Rd,(v,\infty)), v > 0$, and assume the following
uniform integrability condition:
\begin{align*}
\int_0^1 \left( \sup_{(x,s) \in \spctim} H(x,s;v)\right)\,dv < \infty.
\end{align*}
Moreover, for $h(x,s)=f(x)g(s)$ with $f\in C_0(\Rd)$ and $g \in C_c(\R)$ (compact support) define the linear maps
\begin{align} \label{eq:Psi}
\Psi h(x,s)
&:= h(x,s) \gamma(x,s)
+ \int_{v > 0} h(x,s+v) H(x,s;v) \, dv.
\\
\Upsilon h(x,s)
&:= h(x,s) \gamma(x,s)
+ \int_{v>0} \int_{z\in\mathbb R^d} h(x+z, s+v) K(x,s;dz \times (v,\infty))\,dv
\end{align}

Then the CTRW limit process $X_t$ from \eqref{eq:CTRW-limit}
satisfies
\begin{align} \label{eq:lawX}
\int_{t > s} \ex^{x,s} [f(X_t)g(t)]\,dt = U \Psi h(x,s).
\end{align}
and the OCTRW limit
\begin{align*}
\int_{t > s} \ex^{x,s} [f(Y_t)g(t)]\,dt = U \Upsilon h(x,s).
\end{align*}
\end{theorem}

\begin{proof} First note that $H(x,s;v)$ is decreasing to zero on $v \in (0,\infty)$,
for every $(x,s) \in \spctim$, since it
is the tail function of a L\'evy measure. Hence $\Psi h(x,s)\to 0$ as $x,s\to\pm\infty$. Furthermore $\Psi h$ is continuous by the Dominated Convergence Theorem and its support bounded above in $s$. Hence $U\Psi h$ is well defined.

Let $h(x,t) = f(x) g(t)$ for some non-negative $f \in C_0(\Rd)$ and $g \in C_c(\R)$.
Then by Tonnelli's theorem, continuity of Lebesgue measure and the jumps of $X_t$ being countable the left-hand side of \eqref{eq:lawX} equals
\begin{multline*}
\int_{t \in \R} g(t) \ex^{x,s}[f(X_t)] \, dt
= \ex^{x,s}\left [ \int_{t \in \R} g(t) f(X_t) \,dt \right ]
= \ex^{x,s}\left [ \int_{t \in \R} g(t) f(X_{t-}) \,dt \right ]
\\
= \int_{t \in \R} g(t) \ex^{x,s}[f(X_{t-})] \, dt.
\end{multline*}
Now multiply the equation in Theorem 2.3 of \cite{SemiMarkovCTRW} by $g(t)$ (neglecting $Y_t, V_t$ and $R_t$) and integrate over $t \in \R$, to get
\begin{align*}
\int_{t \in \mathbb R} g(t) \mathbf E^{x,s}\left[ f(X_{t-}) \right]\,dt
= \int_{t \in \mathbb R} g(t) \int_{y \in \Rd} f(y) \gamma(y,t) U(x,s;dy,dt)
\\
+\int_{t \in \mathbb R} g(t) \int_{y \in \Rd} \int_{r \in [s,t]} U(x,s;dy,dr) H(y,r;t-r)f(y)\,dt
\end{align*}
Note that we may replace $u^{\chi,\tau}(x,t)\,dx\,dt$ by
$U^{\chi,\tau}(dx,dt)$ in the last equation
on p.1707 of \cite{SemiMarkovCTRW}.
A change of variable then yields \eqref{eq:lawX}.
\end{proof}

\section{A Banach space framework}
\label{sec:banach}

In order to properly define the backwards and forwards equations governing the CTRW limits we establish a Banach space framework on which $U$ is everywhere defined.
Consider $C_0(\R^d\times[a,b))$, the space of bounded continuous functions on $\R^d\times[a,b)$, vanishing at infinity and  $b$ but not necessarily at $a$; i.e. the closure of the space of continuous functions with compact support in $\Rd\times[a,b)$ with respect to the sup norm. The idea is that we will consider the limit process on this space or its dual space for $a\ll s,t\ll b$, where $s$ is the backward variable and $t$ the forward variable.

The crucial observation is that if $f(x,\sigma)=0$ for all $\sigma\ge s$, then, since $D_r$ is strictly increasing,
\begin{equation}\label{projectioneq}Uf(x,\sigma)=0=T_rf(x,\sigma)\end{equation}
 for all $\sigma\ge s$ and $r\ge 0$. This allows us to restrict/project the semigroup $\{T_r\}_{r\ge 0}$ and all of its related operators to $C_0(\R^d\times[a,b))$. In particular, for $\tilde f\in C_0(\R^d\times[a,b))$ pick $ f\in C_0(\spctim)$ such that $\tilde f(x,s)=f(x,s)$ for all $x\in\Rd$ and $s\in[a,b)$ and $f(x,s)=0$ for all $s>b$ and all $s<a-1$. Define the \emph{projection} of $\{T_r\}_{r\ge 0}$ via  $$\tilde T_r\tilde f(x,s):=T_rf(x,s)$$ for all $x\in \Rd$ and $s\in[a,b)$. This is well defined by \eqref{projectioneq} and hence also defines a strongly continuous semigroup with generator $\tilde {\mathcal A}$. Since $\tilde U:\tilde f\mapsto \int_0^\infty \tilde T_r\tilde f\,dr$ is defined for any continuous function with compact support, by Fatou's Lemma it is a bounded operator, and by the resolvent identity, $\tilde U=-\tilde {\mathcal A}^{-1}$. With the same argument, $\tilde\Psi$ is a bounded operator.

 In the following we will not distinguish between $T_r$ and $\tilde T_r$, etc.

\section{Kolmogorov backward equation}

\label{sec:backward}

We now define the transition kernel $P$ for CTRW limits via
\begin{align} \label{eq:P}
\int_{y \in \Rd} f(y) P(x,s; dy, t) = \ex^{x,s}[f(X_t)],
\end{align}
where $f \in C_b(\Rd)$. We interpret the starting point $x$ and starting time $s$ as the backward variables,
and $y$ and $t$ as the forward variables. We also define for $h(x,s)=f(x)g(s),$
\begin{align*}
Ph(x,s) := \int_{\tau > s} \int_{y \in \Rd} P(x,s;dy,\tau) h(y,\tau) \, d\tau
= \int_{\tau > s} g(\tau)\ex^{x,s}[f(X_\tau)]\,d\tau.
\end{align*}

\begin{theorem}[Kolmogorov Backward Equation for CTRW Limits]
\label{th:backwardDE}
Let $h\in C_0(\Rd\times[a,b))$. Then
$Ph$ lies in the domain of $\mathcal A$, and $Ph$ is the unique solution to the problem of finding $v\in C_0(\Rd\times[a,b))$ satisfying
\begin{align*}
-\mathcal A v = \Psi h.
\end{align*}
\end{theorem}

\begin{proof}
For $h(x,s)=f(x)g(s)$, the statement follows directly by adapting the statement of  Proposition~\ref{prop:law} onto  $C_0(\R^d\times[a,b))$. For general $h\in C_0(\Rd\times[a,b))$ the statement follows from the closedness of $\mathcal A$, boundedness of $\Psi$ and the fact that functions of the form $f(x)g(s)$ are a total set. Uniqueness follows from the fact that $\mathcal A$ has the bounded inverse $-U$.
\end{proof}

\begin{remark}
Recall that in $P(x,s;dy,t)$, we call $(x,s)$ the ``backward'' variables and
$(y,t)$ the ``forward'' variables.  Unlike most backward equations,
Th~\ref{th:backwardDE} does not directly relate the $s$-derivative of the
transition kernel $P$ to the generator of spatial motion (acting on $x$). However considering the limit of solutions $P h_n$ with $h_n(x,s)=f(x)g_n(s)$ and $g_n\to\delta_t$ with $\mathrm{supp}(g_n)\subset(t,t+1)$, by the right continuity of $X_t$ and \eqref{eq:P},  $$Ph_n(x,s)\to E^{x,s}f(X_t).$$
\end{remark}

\begin{remark}
\cite{Carmi2010} derive a ``backward fractional Feynman-Kac'' equation, in the
case where jumps have finite variance and are independent of the waiting times.
In its generality, Th~\ref{th:backwardDE} above appears to be new.
\end{remark}

\section{Fokker--Planck Equation}

\label{sec:forward}
In this section we
show that the probability law of the CTRW limit is a unique
solution to a FPE as long as the tail of the temporal L\'evy measure is time independent or the corresponding operator is invertible.
In particular, we are interested in formulating the problem that is solved by the law of $X_t$ given that $X_s=\mu$.

Recall that by the Riesz Representation Theorem the dual space of $ C_0(\R^d\times[a,b))$ is the space of regular bounded measures $\mathcal M( \R^d\times[a,b))$ \citep{rudin1987real}
and that the adjoint of a densely defined linear operator $A$ on a Banach space $X$  is a uniquely defined closed operator on its dual $X^*$. It is defined via
$x^*\in{\mathrm {Dom}}(A^*)$ if there exists $y^*\in X^*$ such that $x^*(Ax)=y^*(x)$ for all $x\in {\mathrm {Dom}}(A)$, and then $A^*x^*=y^*$ \citep{Phillips1955}.
This is relevant as
$$ P^*(\mu \otimes \delta_s)(dy,dt)=\int_{x\in\Rd} P(x,s;dy,t)\mu(dx)\,dt, \quad t \ge s$$
is the quantity of interest (its right-continuous version).

As $U$ and $\Psi$ are bounded operators, so are $U^*$ and $\Psi^*$. In particular, a simple substitution shows that
$$\Psi^*h(dy,dt)=h(dy,dt)\gamma(y,t)+dt\int_{a\le\sigma<t} h(dy,d\sigma)H(y,\sigma;t-\sigma).$$
As is common, we define the convolution $\star$ in the variable $t$ to be
$$(\mu \star_t \nu)(dx,dt) = \int_{s \in [a,b)} \mu(x,dt-s) \nu(dx,ds)$$
for every $\nu \in \mathcal M(\Rd \times [a,b))$
and family of measures $\{\mu(x,dt)\}_{x \in \Rd}$ on $\R$
such that $x \mapsto \mu(x, B)$ is measurable for every Borel set $B \subset \R$.

\begin{proposition}
  If $\gamma(y,t)=\gamma(y)$ and $H(y,t;v)=H(y;v)$ do not depend on $t$, then $\Psi^*$ is one-to-one and
  $$(\Psi^*)^{-1}h=\frac{d}{dt}V\star_t h$$ for $h$ in the range of $\Psi^*$. The Laplace transform of the measure $V(y,\cdot)$ is given by
  $$\int_0^\infty e^{-\lambda t}V(y,dt)=\frac1\lambda \frac{1}{\gamma(y)+\hat H(y,\lambda)} .$$
\end{proposition}
\begin{proof}
  The measures $V(y,\cdot)$ exist since they are renewal measures of subordinators with (fixed) drift $\gamma(y)$ and L\'evy measure $h(y;dw)$ \citep{Bertoin04}. The statement then follows from basic Laplace transform theory.
\end{proof}

\begin{theorem}   [Fokker-Planck Equation for CTRW Limits]
\label{th:FFPE}
Assume $\Psi^*$ is one-to-one. 
Let the initial condition $h$ be given by
$h(dy,dt) = \mu(dy) \delta_s(dt)$, where $\mu\in \mathcal M(\Rd)$
and $a<s<b$. Then $P^*h$ is the unique solution to the problem of finding $v\in \mathcal M(\Rd\times[a,b))$ satisfying
  $$\mathcal A^*(\Psi^*)^{-1}v=-h.$$
\end{theorem}
\begin{proof}
 On $C_0(\R^d\times [a,b))$, $P\phi=U\Psi\phi$ for all $\phi$. Hence
$ P^*h=\Psi^*U^*h$ and equivalently, $(\Psi^*)^{-1}P^*h=U^*h$
for all $h\in\mathcal M(\R^d\times [a,b))$.  Therefore $(\Psi^*)^{-1}P^*h$ is in the range of $U^*$ and hence $(\Psi^*)^{-1}P^*h\in \mathrm{Dom}(\mathcal A^*)$ and
$$\mathcal A^*(\Psi^*)^{-1}P^*h=\mathcal A^* U^* h=-h.$$
Since $\mathcal A^*$ is invertible, $\mathcal A^*(\Psi^*)^{-1}u=0$ implies $u=0$, which implies uniqueness.
\end{proof}

\begin{corollary}
  The transition kernel $P(x,s;dy,t)$ satisfies
  $$-\mathcal A^*(\Psi^*)^{-1}P(x,s;dy,t)dt=\delta_x(dy)\delta_s(dt).$$
\end{corollary}

\subsection*{The Fokker--Planck operator}
In case that temporal and spatial jumps are uncoupled; i.e., $K$ is concentrated on the axes, that is
\begin{align} \label{eq:decoupK}
K(x,s,dz,dw)=K(x,s;dz \times \{0\})+K(x,s;\{0\} \times dw),
\end{align}
above equation simplifies further as it allows the splitting of $\mathcal A=\mathcal D+\mathcal L$ into a temporal operator $\mathcal D$ and a spatial operator $\mathcal L$. In particular, after integration by parts,
$$\mathcal D f(x,s)=\gamma(x,s) \frac{\partial}{\partial s}f(x,s)+\int_{v>0} \frac{\partial}{\partial s}f(x,s+v)H(x,s;v)\,dv$$
and
\begin{equation*}\begin{split}\mathcal Lf(x,s)=& b^i(x,s) \del_{x_i} f(x,s)
+\frac{1}{2}  a^{ij}(x,s) \del_{x_i} \del_{x_j} f(x,s)\\
&+\int\limits_{z \in \Rd} \left[f(x+z,s)
-f(x,s)-  z^i \1(\|z\| < 1) \del_{x_i} f(x,t)\right]  K(x,s;dz,\{0\})\\
\end{split}\end{equation*}
Identifying $\mathcal Df$ as $\Psi \frac{\partial}{\partial s}f(x,s)$, taking adjoints we obtain $$\mathcal A^*f(x,t)=- \frac{\partial}{\partial t}\Psi^*f(x,t)+\mathcal L^*f(x,t).$$
Hence the governing equation simplifies to
\begin{equation}\label{FKeq}\frac{\partial}{\partial t}P^*h=\mathcal L^*(\Psi^*)^{-1}P^*h+h,\end{equation}
earning $\mathcal L^*$ its designation as \emph{Fokker-Planck} operator.

\begin{remark}
Under the assumption that the law of the CTRW limit has Lebesgue
densities, \eqref{FKeq} is equivalent to Equation (45) in
\cite{Kolokoltsov09}.
\end{remark}


\subsection*{The memory kernel}
The non-Markovian nature of the underlying CTRW limit is
represented by a `memory kernel' as in
\citep{Sokolov2006e}.
Their Equation (8) corresponds to  \eqref{FKeq}
where $(\Psi^*)^{-1}$ ``$= \del / \del t\, M \star_t $ ''.
This identifies the anti-derivative of $(\Psi^*)^{-1}$ as the memory kernel $M(t)$. If the coefficients of
 $\gamma(y,t) = \gamma(y)$ and $H(y,t;w) = H(y;w)$ do not depend on $t$, then  $M=V$.
In many cases the measures $V(y,dt)$ are Lebesgue-absolutely continuous
with density $v(y,t)$; e.g.\ when $\gamma(y) > 0$
\citep[Prop~1.7]{Bertoin04}.

\section{Anomalous Diffusion: Examples} \label{sec:examples}

\subsection{Subdiffusion in a time-dependent potential}

Let $\beta \in (0,1)$ and define
\begin{align*}
H_\beta(w) := \frac{1}{\Gamma(1-\beta)}w^{-\beta}, 
\quad 
h_\beta(w) := -\frac{\del}{\del w} H_\beta(w)
= \frac{\beta}{\Gamma(1-\beta)} w^{-1-\beta}.
\end{align*}
We introduce the scaling parameter $c > 0$, and define 
\begin{align}\label{eq:hbetatau2} 
H_\beta^c(w) := 1 \wedge [H_\beta(w)/c], 
\quad 
h_\beta^c(w) := \mathbf 1\{w > (\Gamma(1-\beta) c)^{-1/\beta}\} h_\beta(w)/c.
\end{align}
Note that $H_\beta^c(w)$ is the tail function of a Pareto law
on $(0,\infty)$, and $h_\beta^c(w)$ is its density.
This law shall be assumed for the distribution of waiting times. 
We also assume
probabilities $\ell(x,t)$ and $r(x,t)$ to jump left or right on a one-dimensional lattice.  A CTRW with such jumps and waiting times may be represented as a Markov chain in $\spctim$, with transition kernel
\begin{align} \label{eq:kernel1}
K^c(x,s;dz,dw) &=
\left[\ell(x,s+w) \delta_{-\Delta x}(dz) + r(x,s+w) \delta_{\Delta x}(dz)\right]
h_\beta^c(w)dw.
\end{align}
Such CTRWs are a useful model for subdiffusive processes, i.e.\ processes whose variance grows slower than linearly \citep{Metzler2000}.
For the limit to exist as $c \to \infty$, we assume
\begin{align}
\ell(x,s) + r(x,s) &= 1, & r(x,s) - \ell(x,s) =  b(x,s) \Delta x.
\end{align}
where $b(x,s)$ is a bias and $\Delta x$ is the lattice spacing.
The bias varies with space and time and is given e.g.\ by the concentration gradient of a chemo-attractive substance, which itself diffuses in space \citep{HenryLanglands10chemo}.

We consider the limit $c \to \infty$, with $(\Delta x)^2 = 1/c$.
The limiting coefficients of $(A_r,D_r)$ are
\begin{gather*}
a(x,s) = 1, \quad  b(x,s) = \text{given} , \quad \gamma(x,s) = 0, \quad
K(x,s;dz,dw) = \delta(dz) h_\beta(w)dw,
\end{gather*}
where
\begin{align} \label{eq:hbeta}
h_\beta(w) = \beta w^{-\beta-1}  \1\{w > 0\} / \Gamma(1-\beta).
\end{align}
and $\delta$ denotes the Dirac measure concentrated at
$0 \in \Rd$.
Apply \citet[Th~IX.4.8]{JacodShiryaev} to see that the convergence \eqref{eq:spctim-conv} holds.
The infinitesimal generator reads
$$\mathcal A f(x,s) = b(x,s) \del_x f(x,s) + \frac{1}{2} \del_x^2 f(x,s)
- \del_{-s}^\beta f(x,s)$$
where $\del_{-s}^\beta f$ denotes the negative fractional derivative \citep{MeerschaertSikorskii,kolokoltsov2011markov}.
Given a suitable ``terminal condition'' $f \in C_b(\spctim)$, the Kolmogorov backward equation is hence
\begin{align*}
\del_{-s}^\beta Pf(x,s)
= b(x,s) \del_x Pf(x,s)
 + \frac{1}{2} \del_x^2 Pf(x,s)
 + \del_{-s}^{\beta-1} f(x,s)
\end{align*}
where the negative Riemann-Liouville fractional integral of order
$\beta > 0$ is denoted by
\begin{align}\label{eq:frac-int}
\del_{-t}^{-\beta} f(t)
:=\frac{1}{\Gamma(\beta)} \int_{r > 0} f(t+r) r^{\beta - 1}\,dr
\end{align}
(see also \cite{BajlekovaThesis}).

For the forward equation, we note that
$H(x,s;w) = H_\beta(w) := w^{-\beta}/\Gamma(1-\beta)$ has Laplace transform $\hat H_\beta(\lambda) = \lambda^{\beta-1}$.
Hence $\hat V(\lambda) = \lambda^{-\beta}$, which inverts to
$V(y,r) = r^{\beta-1}/\Gamma(\beta) = H_{1-\beta}(r)$.
Thus $(\Psi^*)^{-1}$ may be interpreted as the fractional derivative $\del_t^{1-\beta}$.
The adjoint of $\mathcal L$ is given by
$$\mathcal L^*f(dy,dt) = -\del_y[b(y,t)f(dy,dt)] + \frac{1}{2} \del_y^2 f(dy,dt),$$
hence the distributional Fokker--Planck equation is
\begin{align*}
\del_t P^*[\mu \otimes \delta_s] &=
-\del_y \left[b\, \del_t^{1-\beta} P^*(\mu \otimes \delta_s) \right]
+ \frac{1}{2} \del_y^2 \del_t^{1-\beta} P^*(\mu \otimes \delta_s)
+ \mu \otimes \delta_s
\end{align*}
(compare \cite{HLS10PRL}).

\begin{remark}
The coefficients $a$, $b$, $\gamma$ and $K$ above match the coefficients of the
stochastic differential equation (7) in \cite{Magdziarz2014}
where the diffusivity $ =1$.
The Fokker--Planck equation also matches their equation (6).
A CTRW scaling limit whose diffusivity varies in space and time
is achieved e.g.\ if \eqref{eq:kernel1} is replaced by
$$K^c(x,s;dz,dw) = \mathcal N(dz | c^{-1/2} b(x,s), c^{-1} a(x,s)) h_\beta^c(w)\,dw,$$
where $\mathcal N(dz | m, s^2)$ denotes a univariate Gau\ss{}ian distribution 
with mean $m$ and variance $s^2$.
\end{remark}

\subsection{Traps of spatially varying depth}

\cite{Fedotov2012} study CTRWs with spatially varying ``anomalous exponent'' $\beta(x) \in (0,1)$.
They find that in the long-time limit the (lattice) CTRW process is localized at the lattice point where $\beta(x)$ attains its minimum, a phenomenon termed ``anomalous aggregation''.
Using flux balances, \cite{Chechkin2005a} derive a fractional diffusion equation
with a ``variable order'' Riemann-Liouville derivative, which we can now
rephrase in our framework.
In this example, we assume unbiased jumps of probability $1/2$ to the left and right,
and fix a Lipschitz continuous function $\beta(x) \in (\varepsilon, 1-\varepsilon)$ for some $\varepsilon > 0$. The waiting time at each lattice site has the density
$h_{\beta(x)}^c(w)$ as in \eqref{eq:hbetatau2}, with $\beta$ replaced by $\beta(x)$.
In the limit $c \to \infty$ with $(\Delta x) = 1/c$ we arrive at the coefficients
\begin{gather} \label{eq:ex2coeff}
a(x,s) = 1, \quad  b(x,s) = 0, \quad \gamma(x,s) = 0, \quad
 K(x;dz \times dw) = \delta_0(dz) h_{\beta(x)}(w)dw.
\end{gather}
As mentioned in \citet[p.272]{Bass1988}, the standard Lipschitz continuity and growth assumptions guarantee the existence and uniqueness of a strong (pathwise) solution to a stochastic differential equation with generator $\mathcal A$ given by \eqref{eq:DAgen} and \eqref{eq:ex2coeff}. The negative fractional derivative of variable order $\beta(x)$ is
\begin{align*}
\del_{-t}^{\beta(x)} f(x,t)
&= \int_{w > 0} [f(x,t) - f(x,t+w)]h_{\beta(x)}(w)\,dw,
\end{align*}
where $h_{\beta(x)}(w)$ is as in \eqref{eq:hbeta}, with $\beta$ dependent on $x$.
As in the previous example, we have 
$V(y,r) = r^{\beta(y)-1} / \Gamma(\beta(y))$, and
the Kolmogorov backward equation hence reads
\begin{align*}
\del_{-s}^{\beta(x)} Pf(x,s) = \frac{1}{2} \del_x^2 Pf(x,s)
+ \del_{-s}^{\beta(x)-1} f(x,s)
\end{align*}
and the FPE
\begin{align*}
\del_t P^*(\mu \otimes \delta_s) = -\del_y^2 \left [\del_t^{1-\beta(y)} P^*(\mu \otimes \delta_s)\right ] + \mu \otimes \delta_s.
\end{align*}

\begin{remark}
A different approach to spatially varying traps is taken in
\cite{Kolokoltsov09}. There, the generator
\begin{align*}
\mathcal A f(x,s)
&= \int_0^\infty \int_{S^{d-1}} (f(x+y,s)-f(x,s))
\frac{d|y|}{|y|^{1+\alpha}} S(x,s,\bar y)\, d_S \bar y
\\
&+ \frac{w(x,s)}{\Gamma(-\beta)}\int_0^\infty (f(x,s+v)-f(x,s))\frac{1}{v^{1+\beta}}
\,dv
\\
&=: \mathcal L f(x,s)
+ w(x,s) \partial_{-s}^\beta f(x,s)
\end{align*}
for the process $(A_r,D_r)$ is assumed, where
$\alpha \in (0,2)$, $\beta \in (0,1)$,
$\bar y = y / |y|$, $S(x,s,\bar y)\,d_S\bar y$ is a symmetric
Lebesgue-absolutely continuous measure
on the unit sphere and $w$ a measureable function.
The scaling limit process is explicitly constructed.
An application of Theorem 4.3 therein gives
\begin{align}
- \partial_t P(y,t) = \partial_t^\beta \left[ w(y,t) U(y,t) \right]
\end{align}
where $U(y,t)$ is the density of the potential measure of the
Feller process $(A_r,D_r)$ and $P(\cdot, t)$ the probability density
of $X_t$.
Assuming that $w(x,s) = w(x)$ does not depend on the time
variable, we may go one step further and write the FPE for this CTRW limit process as
\begin{align*}
\del_t P^*[\mu(dy)\otimes \delta_s(dt)]
= \mathcal L^* \frac{1}{w(y)} \del_t^{1-\beta} P^* [\mu(dy)\otimes \delta_s(dt)] + \mu(dy)\otimes \delta_s(dt);
\end{align*}
note that, unlike in the previous example, we now have
$$V(y,dt) = \frac{t^{\beta-1}}{\Gamma(\beta) w(y)}\,dt,$$
and $\mathcal L$ is formally self-adjoint.
\end{remark}
The Kolmogorov backward equation reads
\begin{align*}
-\mathcal L P f(x,s) - w(x,s) \partial_{-s}^\beta P f(x,s)
= w(x,s) \partial_{-s}^{\beta-1} f(x,s),
\end{align*}
where $w(x,s)$ may be time-dependent.

\subsection{Space- and time-dependent L\'evy Walks}

The standard L\'evy Walk consists of i.i.d.\ movements with constant speed, where directions are drawn from a probability distribution $\lambda(d\theta)$ on the unit sphere $S^{d-1}$ in $\Rd$ and movement lengths are drawn from a probability distribution which lies in the domain of attraction of a stable law, e.g.\ $h^c_\beta(w)$ \eqref{eq:hbetatau2}.
We consider the case $\beta \in (0,1)$, which is termed ``ballistic'' since the second moment grows quadratically \citep{Sokolov_Book}.
Coupled CTRWs, in which waiting times of length $W_k$ come with jumps of size
$|J_k| = W_k$, serve as an approximation of a L\'evy Walk with velocity $1$.

In this example, we consider a CTRW approximation of a L\'evy Walk with space- and time-dependent drift $b(x,s)$.
Such a CTRW is given by the Markov chain with transition kernel
\begin{align}
K^c(x,t;B\times I)
= \int_{\theta \in S^{d-1}} \int_{r > 0} \1_B(
r\theta + b(x,t)/c) \1_I(
r)
h_\beta^c(r)dr\, \lambda(d\theta),
\end{align}
which converges to a limiting space-time process $(A_r, D_r)$ with generator \eqref{eq:DAgen} and coefficients
\begin{align*}
a &= 0, & b^i(x,s) &= {\rm given}, & \gamma(x,s) &= 0,
\end{align*}
\begin{align*}
K(x,s;B \times I)
= K(B \times I)
= \int_{\theta \in S^{d-1}} \int_{r > 0} \1_B(
r\theta) \1_I(
r)\, h_\beta(r)\,dr\,\lambda(d\theta).
\end{align*}
(Note that here $b(x,s)$ is relative to there being no cut-off function $\1(\|z\| < 1)$ in \eqref{eq:DAgen}.)
The infinitesimal generator has the pseudo-differential representation \citep{Jurlewicz,triCTRW}
\begin{align*}
\mathcal A f(x,s)
&= b^i(x,s) \del_{x_i} f(x,s)
+ \int_{\theta \in S^{d-1}} \int_{w > 0}\left [f(x+w\theta, s+w) - f(x,s)\right ] h_\beta(w)\,dw\,\lambda(d\theta)
\\
&= b^i(x,s) \del_{x_i} f(x,s)
- \int_{\theta \in S^{d-1}} \left (-\langle \theta, \nabla_x\rangle - \del_s\right )^{\beta} f(x,s) \lambda(d\theta).
\end{align*}
The Kolmogorov backwards equation for the CTRW scaling limit is thus
\begin{align*}
b^i(x,s) \del_{x_i} Pf(x,s)
- \int_{\theta \in S^{d-1}}
\left (- \langle \theta, \nabla \rangle - \del_{-s}\right )^{\beta} Pf(x,s)\, \lambda(d\theta) = \del_{-s}^{\beta-1}f(x,s).
\end{align*}
As $H(x,s;w) = w^{-\beta} / \Gamma(1-\beta)$ as in Example 6.1,
the governing FPE is
\begin{align*}
\mathcal A^* \del_t^{1-\beta} P^*[\mu \otimes \delta_s](dy,dt)
= -\mu(dy) \otimes \delta_s(dt)
\end{align*}
The generator $\mathcal A$ does not have a decomposition into
$\mathcal L + \mathcal D$ as in \eqref{eq:decoupK}, and hence
we stop here.

\subsection*{Acknowledgements}
The authors thank Prof.\ Ren\'e Schilling, Prof.\ Mark Meerschaert and Prof.\ Atmah Mandrekar for their helpful advice in preparing this manuscript.
B.~Baeumer was supported by the Marsden Fund Council from Government funding, administered by the Royal Society of New Zealand.
P.~Straka was supported by the UNSW Science Early Career Research Grant and
the Australian Research Council's Discovery Early Career Research Award.

\bibliographystyle{abbrvnat}

\end{document}